\documentclass[a4paper,11pt]{article}
\usepackage{graphicx,xcolor}
\usepackage{float}
\usepackage{amsmath,amssymb,bm,amsthm}
\usepackage{ascmac}
\usepackage{array}
\usepackage{hyperref}
\usepackage[hang,small,bf]{caption}
\usepackage[subrefformat=parens]{subcaption}
\usepackage{natbib}
\usepackage{enumitem}

\bibpunct[:]{(}{)}{,}{a}{}{,}
\hypersetup{
  colorlinks=true,
	citecolor=blue,
	linkcolor=red,
	urlcolor=orange,
}
\title{Asymptotic Expansion of the Kallianpur–Striebel Formula under Perturbations of Linear Models}
\author{Masahiro Kurisaki\thanks{AIP Center, RIKEN\protect\\ email: \texttt{masahiro.kurisaki@riken.jp}}~\thanks{Japan Science and Technology Agency CREST \protect\\ The author was supported by Japan Science and Technology Agency CREST JPMJCR2115, and JSPS KAKENHI Grant Number JP24KJ0667.}}

\numberwithin{equation}{section}
\allowdisplaybreaks[3]
\newtheorem{theorem}{Theorem}[section]
\newtheorem{proposition}[theorem]{Proposition}
\newtheorem{lemma}[theorem]{Lemma}

\theoremstyle{definition}
\newtheorem{definition}[theorem]{Definition}

\theoremstyle{remark}
\newtheorem{remark}{Remark}

\begin{document}
\maketitle
\begin{abstract}
  In this paper, we study a nonlinear state-space model represented as a perturbation of a linear model and investigate an asymptotic expansion of the nonlinear filter with respect to the perturbation parameter. The conditional expectation of the hidden state given the observations admits a closed-form representation via the Kallianpur–Striebel formula. We provide a rigorous justification of the expansion of this representation in probability under very general assumptions.

  Our framework applies not only to nearly linear models but also to more general situations, including nonlinear systems with small system noise. In these cases, the coefficients of the resulting expansion can be computed through systems of ordinary differential equations. Although the explicit computational procedure will be presented in a subsequent paper, the present work establishes the theoretical foundation of our asymptotic expansion approach to nonlinear filtering.
\end{abstract}


\section{Introduction}

In probability theory, the filtering problem refers to the problem of calculating the conditional expectation $E[X_t|\mathcal{Y}_t]$ for an unobservable process $\{X_t\}_{t\geq 0}$, and the filtration $\{\mathcal{Y}_t\}_{t\geq 0}$ generated by an observation process $\{Y_t\}_{t\geq 0}$. In the standard continuous-time setting, $X$ and $Y$ are typically given by the solution of stochastic differential equations
\begin{align*}
  &dX_t=\alpha(X_t)dt+\beta(X_t)dV_t,\\
  &dY_t=h(X_t)dt+\sigma(t)dW_t,
\end{align*}
where $V$ and $W$ are independent Brownian motions, and $\alpha$, $\beta$ and $h$ are some known functions.

In general, the above filtering problem does not admit a closed-form solution, and the conditional distribution of $X_t$ given $\mathcal{Y}_t$ is infinite-dimensional. An important exception is the linear Gaussian case, where equations of $X$ and $Y$ are given in the form
\begin{align*}
  &dX_t=a(t)X_t dt + b(t)dV_t,\\
  &dY_t=c(t)X_t dt+\sigma(t)dW_t.
\end{align*}
In this case, the conditional distribution of $X_t$ given $\mathcal{Y}_t$ is Gaussian, and the conditional mean and covariance are given as the recursive formula according to the Kalman-Bucy filter \citep{Liptser2001-2,doi:10.1137/S0040585X97T99037X,RUTKOWSKI1993377}.

In the general nonlinear case, it is difficult to obtain an analytical solution to the
filtering problem, and various approximation methods have therefore been proposed.
For instance, the extended Kalman filter and its variants \citep{picard1991, 882463,847749,2003OcDyn..53..343E} rely on local linear or Gaussian
approximations, PDE-based approaches \citep{lototsky2011chaos} aim at approximating the associated stochastic partial
differential equations, and particle filters \citep{4378823,6530707} approximate the conditional distribution through Monte Carlo simulations.

Another important direction in nonlinear filtering is the perturbation approach, where one studies the asymptotic behaviour of the filter around a tractable reference model.

For example, \citet{picard1986} analysed the high signal-to-noise regime and derived finite-dimensional approximations of the filter in the small-noise limit. Although rigorous error bounds are obtained, the method provides a specific approximation rather than a systematic expansion, and it is restricted to one-dimensional models with strong structural assumptions such as the injectivity of the observation function.

On the other hand, \citet{fujii2014momentum} proposed a momentum-space asymptotic expansion based on the Zakai equation and Fourier techniques. While this approach yields a recursive structure for the expansion, its mathematical justification remains largely open. Moreover, the resulting scheme requires solving a system of ordinary differential equations for each Fourier mode, and therefore does not fully overcome the curse of dimensionality.

Another line of research on asymptotic expansions for conditional expectations was developed in \citet{YOSHIDA200353, MASUDA200437}. These works derive asymptotic expansions of the joint distribution of
\[
  (X_{t_1}^\epsilon,\cdots,X_{t_m}^\epsilon,Y_{s_1}^\epsilon,\cdots,Y_{s_n}^\epsilon),
\]
where $t_1,\cdots,t_m,s_1,\cdots,s_n$ are fixed time points. Based on this joint density expansion, one can approximate conditional expectations of the form
\[
  E[\phi(X_{t_1}^\epsilon,\cdots,X_{t_m}^\epsilon)\mid Y_{s_1}^\epsilon,\cdots,Y_{s_n}^\epsilon].
\]
This approach provides an efficient method for computing conditional expectations given finite-dimensional observations. However, it is different in nature from the standard recursive filtering framework, where the estimate is updated sequentially as new observations arrive.

In this paper, we develop a different asymptotic expansion framework for nonlinear filtering. While our approach is related in spirit to these earlier expansion-based methods, it is based on a direct expansion of the Kallianpur--Striebel formula rather than on an expansion of the finite-dimensional joint density. This formulation is better suited to the recursive nature of continuous-time filtering, and our main objective is to establish a fully rigorous justification of the expansion under general assumptions.

In our setting, we consider the linear signal and perturbed observation model
\begin{align*}
dX_t &= a(t)X_t dt + b(t)dV_t,\\
dY_t^\epsilon &= \left\{ c(t)X_t + \sum_{i=1}^\infty H_t^{[i]} \epsilon^i \right\}dt + \sigma(t)dW_t,
\end{align*}
where $\epsilon>0$ is a small parameter. The objective of this work is to provide a rigorous justification of the asymptotic expansion
\begin{align}
\label{eq-expansion-intro}
E[f(X_t)\mid \mathcal{Y}_t^\epsilon]
= m_t^{[0]} + m_t^{[1]}\epsilon + \cdots + m_t^{[n]}\epsilon^n+O(\epsilon^{n+1}),
\end{align}
where ${\mathcal{Y}_t^\epsilon}$ denotes the filtration generated by ${Y_t^\epsilon}$.

The motivation for this formulation is twofold. First, a wide class of nonlinear filtering problems can be reduced to this framework through suitable perturbative transformations. In particular, nonlinear small-noise models of the form
\begin{align*}
dX_t^\epsilon &= \alpha(X_t^\epsilon)dt + \epsilon \beta(X_t^\epsilon)dV_t,\\
dY_t^\epsilon &= h(X_t^\epsilon)dt + \sigma(t)dW_t,
\end{align*}
with smooth nonlinear coefficients can be treated within our framework, as shown in the next section.

Second, the Kallianpur–Striebel formula provides a natural representation of the filter in terms of weighted conditional expectations. Formally expanding this representation yields explicit expressions for the expansion coefficients, which can be computed recursively by finite-dimensional ordinary differential equations. This feature leads to a tractable and implementable framework, and will be developed in a subsequent work.

However, providing a rigorous justification of the expansion remains nontrivial. The main difficulty lies in controlling the remainder term in \eqref{eq-expansion-intro}. The martingale process appearing in the Kallianpur–Striebel formula typically possesses only an $L^1$-property, and establishing $\epsilon$-uniform estimates for the weighted conditional expectations requires delicate analysis. This difficulty has not been fully addressed in the existing expansion-based approaches. 

For example, \citet{fujii2014momentum} studies a similar type of expansion, and discusses it mainly at a formal level. The paper explicitly notes that providing a rigorous justification is beyond the scope of that work. This reflects the intrinsic difficulty of rigorous justification of expansion-based approaches to nonlinear filtering.

The primary contribution of this paper is to overcome this difficulty and to establish a mathematically rigorous probabilistic error estimate under mild and general conditions. This result provides a solid theoretical foundation for systematic perturbation methods in nonlinear filtering. 

The rest of this paper is organized as follows. In Section \ref{section-framework}, we present our framework and assumptions precisely, and a motivating example is given in Section \ref{section-example}. This example illustrates that our framework accommodates general nonlinear models with small system noise.

Section \ref{section:asymptotic-expansion} is devoted to presenting an asymptotic expansion formula for the conditional expectation. We first recall the representation formula of the conditional expectation via the Kallianpur–Striebel formula and consider the expansion of the exponential martingale in Section \ref{section:KS-formula}. A rigorous estimate for the remainder term is then established in Section \ref{section:estimation-residue}.

\section{Framework and examples}\label{section-framework}
Let $(\Omega,\mathcal{F},\{\mathcal{F}_t\}_{t\geq 0},P)$ be a filtered probability space, and we consider the following perturbed linear system for \( 0 \leq \epsilon < 1 \):
\begin{align}
  \label{eq-X-linear}dX_t &= a(t) X_t \, dt + b(t) \, dV_t, \\
  \label{eq-Y-epsilon}dY_t^\epsilon &= H_t^\epsilon \, dt + \sigma(t) \, dW_t,~Y_0=0,
\end{align}
where \( a \), \( b \), \( c \), and \( \sigma \) are measurable functions taking values in \( M_{d_1}(\mathbb{R}) \), \( M_{d_1,m_1}(\mathbb{R}) \), \( M_{d_2,d_1}(\mathbb{R}) \), and \( M_{d_2,m_2}(\mathbb{R}) \), respectively. Here, \( \{V_t\} \) and \( \{W_t\} \) are independent \( m_1 \)- and \( m_2 \)-dimensional \( \{\mathcal{F}_t\} \)-Brownian motions. We also assume that \( X_0 \) is normally distributed and independent of \( \{V_t\} \) and \( \{W_t\} \), so that \( \{(X_t, Y_t)\} \) forms a \( d_1 + d_2 \)-dimensional Gaussian process.

We also assume that \( H_t^\epsilon \) is a \( d_2 \)-dimensional progressively measurable process and is assumed to admit an expansion
\begin{align*}
  H_t^\epsilon = c_0(t)+c(t) X_t + \sum_{i=1}^{n-1} H_t^{[i]} \epsilon^i + r_t^{\epsilon, n}
\end{align*}
for every $n \in \mathbb{N}$.

Moreover, we assume the following conditions.
\begin{enumerate}[label=\textbf{[A\arabic*]}]
  \item for every \( t \geq 0 \),
\begin{align}
  \label{eq-assumption-a} &\int_0^t |a(s)| \, ds < \infty, \\
  \label{eq-assumption-b} &\int_0^t |b(s)|^2 \, ds < \infty, \\
  \label{eq-assumption-c0} &\int_0^t |c_0(s)| \, ds < \infty,\\
  \label{eq-assumption-c} &\int_0^t |c(s)|^2 \, ds < \infty, \\
  \label{eq-assumption-sigma-1} &\int_0^t |\sigma(s)|^2 \, ds < \infty,
\end{align}
and that there exists a constant \( C > 0 \) such that for every \( t \geq 0 \),
\begin{align}
  \label{eq-assumption-sigma-2} \lambda_{\min}(\sigma(t) \sigma(t)^\top) > C,
\end{align}
where \( \lambda_{\min}(\sigma(t) \sigma(t)^\top) \) denotes the smallest eigenvalue of \( \sigma(t) \sigma(t)^\top \).
Here, $|A|$ denotes the Frobenius norm for a matrix $A$. \label{assumption-A0}
  \item For any $t\geq 0$, there exists $p\geq 1$, $C>0$ and $\delta>0$ such that 
  \begin{align*}
    \int_0^t E[|H_s^{\epsilon_1}-H_s^{\epsilon_2}|^p]ds \leq C|\epsilon_1-\epsilon_2|^{1+\delta}
  \end{align*}
  for any $\epsilon_1,\epsilon_2 \in [0,1)$.\label{assumption-A1}
  \item For each $i=1,2,\cdots$, $\{H_t^{[i]}\}_{t\geq 0}$ is a progressively measurable process which is independent of $\{W_t\}_{t\geq 0}$, and satisfies
  \begin{align*}
    \int_0^t E[|H_s^{[i]}|^2] ds<\infty
  \end{align*}
  for every $t\geq 0$.\label{assumption-A2}
  \item For $0<\epsilon<1$, $\{r_t^{n,\epsilon}\}_{t\geq 0}$ is a progressively measurable process which is independent of $\{W_t\}_{t\geq 0}$ and satisfies  for some $p\geq 2$
  \begin{align*}
    \sup_{0< \epsilon <1} \frac{1}{\epsilon^{n}}E\left[\sup_{0\leq s\leq t}|r_s^{n,\epsilon}|^p\right]^\frac{1}{p} <\infty
  \end{align*}
  for every $t\geq 0$ and $n \in \mathbb{N}$.\label{assumption-A3}
\end{enumerate}

Furthermore, let \( \{U_t^\epsilon\}_{t \geq 0,~0 \leq \epsilon < 1} \) be a family of random variables, and assume the following conditions.
\begin{enumerate}[label=\textbf{[A\arabic*]}] \setcounter{enumi}{4} 
  \item For any $n=0,1,2,\cdots$, $U_t^\epsilon$ has the expression 
  \begin{align*}
    U_t^\epsilon=\sum_{i=0}^{n-1}U_t^{[i]}\epsilon^i+u_t^{n,\epsilon},
  \end{align*}
  where each $U_t^{[i]}$ and $u_t^{n,\epsilon}$ are independent of $\{W_s\}_{0\leq s\leq t}$, and satisfy
  \begin{align*}
    &\sup_{0\leq s\leq t}E[|U_s^{[i]}|^p]<\infty,\\
    &\sup_{0< \epsilon <1}\sup_{0\leq s\leq t} \frac{1}{\epsilon^{n}}E\left[|u_s^{n,\epsilon}|^p\right]^\frac{1}{p} <\infty
  \end{align*}
  for every $t\geq 0$, $i=0,1,2,\cdots$, $n \in \mathbb{N}$ and $p\geq 1$.\label{assumption-A6}
\end{enumerate}

The goal of this work is to provide a rigorous justification for an asymptotic expansion of the conditional expectation $E[U_t^\epsilon \mid \mathcal{Y}_t^\epsilon]$ where $\{\mathcal{Y}_t^\epsilon\}$ is the usual augmentation with null sets of the filtration generated by $\{Y_t^\epsilon\}$. Before proceeding to the expansion formula, we present a motivating example of our framework.

\section{Motivating Example}\label{section-example}
Let $\{X_{t}^\epsilon\}_{t\geq 0}$ and $\{Y_t^\epsilon\}$ be a $d_1$ and $d_2$-dimensional processes satisfying
\begin{align}
  \label{eq-eta}&dX_t^\epsilon=\alpha(X_t^\epsilon)dt+\epsilon \beta(X_t^\epsilon)dV_t,~~X_0=x_0,\\
  \label{eq-Y-theta}&dY_t^\epsilon=h(X_t^\epsilon)dt+\sigma(t)dW_t,~~Y_0=0.
\end{align}
where $\alpha:\mathbb{R}^{d_1}\to \mathbb{R}^{d_1}$, $\beta:\mathbb{R}^{d_1}\to M_{d_1,m_1}(\mathbb{R})$ and $h:\mathbb{R}^{d_1}\to \mathbb{R}^{d_2}$ are of class $C^\infty$, and $x_0 \in \mathbb{R}^{d_1}$ is a constant. 

Furthermore, we assume that for any $k \in \mathbb{Z}_+$
\begin{align*}
  \sup_{x \in \mathbb{R}^{d_1}}|\partial^k\alpha(x)|+\sup_{x \in \mathbb{R}^{d_1}}|\partial^k\beta(x)|<\infty,
\end{align*}
and there exists a constant $C$ and $q$ such that
\begin{align}
  \label{eq-assumption-h}\sup_{x \in \mathbb{R}^{d_1}}|\partial^kh(x)|\leq C(1+|x|^q).
\end{align}
Here, for $\phi \in C^k(\mathbb{R}^m; \mathbb{R}^n)$, the $k$-th derivative $\partial^k \phi(x)$ is regarded as an $\mathbb{R}^n$-valued $k$-tensor whose $(i_1,\ldots,i_k)$-th component is given by
  \begin{align*}
      \frac{\partial}{\partial x_{i_1}} \cdots \frac{\partial}{\partial x_{i_k}} \phi(x_1, \cdots, x_m).
  \end{align*}  
For $t>0$ and $p \ge 1$, introduce the norm
\begin{align*}
  \|\xi\|_{p,t}=E\left[ \sup_{0\leq s\leq t}|\xi_s|^p \right]^\frac{1}{p}.
\end{align*}   
for any $\mathbb{R}^{d_1}$-valued continuous process $\xi$.
Then, the following result holds.

\begin{proposition}\label{prop-expansion-of-SDE}
  For any fixed $t>0$ and $p \ge 1$, the mapping $\epsilon \mapsto X^\epsilon$ is continuously differentiable in the Banach space of continuous processes endowed with the norm $\|\cdot\|_{p,t}$. 
  
  Moreover, the derivatives $X_t^{[k],\epsilon} := \partial_\epsilon^k X_t^\epsilon~(k\geq 1)$ satisfy the stochastic differential equation obtained by formal differentiation of \eqref{eq-eta}:
 \begin{align*}
  &dX_t^{[k],\epsilon}=\sum_{j=1}^k \sum_{(i_1,\cdots,i_j)\in \Lambda(k,j)}\nu(k,(i_1,\cdots,i_j))\partial^{j}\alpha(X_t^\epsilon)[X_t^{[i_1],\epsilon}\otimes \cdots \otimes X_t^{[i_j],\epsilon}]dt\\
  &+\epsilon\sum_{j=1}^k \sum_{(i_1,\cdots,i_j)\in \Lambda(k,j)}\nu(k,(i_1,\cdots,i_j))\partial^{j}\beta(X_t^\epsilon)[X_t^{[i_1],\epsilon}\otimes \cdots \otimes X_t^{[i_j],\epsilon}]dV_t\\
  &+k\sum_{j=1}^{k-1} \sum_{(i_1,\cdots,i_j)\in \Lambda(k-1,j)}\nu(k-1,(i_1,\cdots,i_j))\partial^{j}\beta(X_t^\epsilon)[X_t^{[i_1],\epsilon}\otimes \cdots \otimes X_t^{[i_j],\epsilon}]dV_t.
\end{align*}
with $X_0^{[k],\epsilon}=0$. Here,
\begin{align*}
  \Lambda(k,j)=\left\{ (i_1,\cdots,i_j)\in \mathbb{N}^j;i_1\leq \cdots \leq i_j,~i_1+\cdots+i_j=k  \right\},
\end{align*}
and $\nu(k,(i_1,\cdots,i_j))$ represents the number of partitions of the set $\{1,\cdots,k\}$ into $j$ subsets with $(i_1,\cdots,i_j)$ elements, and 
\begin{align*}
  &\partial^{j}\alpha(x_1,\cdots,x_{d_1})[\xi^1\otimes \cdots \otimes \xi^k]\\
  =&\sum_{i_1,\cdots,i_k=1}^{d_1}\frac{\partial}{\partial x_{i_1}} \cdots \frac{\partial}{\partial x_{i_k}} \alpha(x_1, \cdots, x_{d_1})\xi_{i_1}^1\cdots \xi_{i_k}^k
\end{align*}
for $\xi^{i}=(\xi_1^i,\cdots,\xi_{d_1}^i)~(i=1,\cdots,k)$.
\end{proposition}
The proof of this proposition for the first derivative ($k=1$) follows from Theorem 3.3.2 in \citet{kunita2019stochastic}. Higher-order derivatives can be obtained by induction. Indeed, since the coefficients and their derivatives are bounded, the required $L^p$-estimates follow from the usual argument with the Burkholder–Davis–Gundy inequality and Gronwall’s lemma.

As a consequence of Proposition \ref{prop-expansion-of-SDE}, $X^{[k]}=X^{[k],0}$ ($k\geq 0$) satisfy the equations
\begin{align}
&dX_t^{[0]}=a(X_t^{[0]})dt,~~X_0^{[0]}=x_0,\\
&dX_t^{[1]}=\partial\alpha(X_t^{[0]})[X_t^{[1]}]dt+\beta(X_t^{[0]})dV_t,~~X_0^{[1]}=0.\label{eq-X-1}
\end{align}
These equations show that $X_t^{[0]}$ is deterministic, while $X_t^{[1]}$ is the solution of the linear equation (\ref{eq-X-1}). Hence, if we introduce $\tilde{X}_t=(X_t^{[1]},V_t)^\top$, then $\tilde{X}_t$ satisfies the linear equation
\begin{align*}
d\tilde{X}_t=\begin{pmatrix}
\partial \alpha(X_t^{[0]})&O\\
O&O
\end{pmatrix}\tilde{X}_tdt+\begin{pmatrix}
\beta(X_t^{[0]})\ I_{m_1}
\end{pmatrix}dV_t.
\end{align*}
Furthermore, by Taylor's theorem, we obtain the expansion
\begin{align}
\label{eq-h-theta}h(X_t^\epsilon)=h(X_t^{[0]})+\partial h(X_t^{[0]})X_t^{[1]}\epsilon+\sum_{i=2}^{n-1} \rho_t^{[i]}(h)\epsilon^i+\epsilon^{n}\int_0^1 \rho_t^{[n],\epsilon u}(h)u^{n-1}du,
\end{align}
where
\begin{align}
\label{eq-def-rho}\rho_t^{[k],\epsilon}(h)=\frac{1}{k!}\sum_{j=1}^k\sum_{(i_1,\cdots,i_j)\in \Lambda(k,j)}\nu(k,(i_1,\cdots,i_j))\partial^jh(X_t^{\epsilon})[X_t^{[i_1],\epsilon}\otimes \cdots \otimes X_t^{[i_j],\epsilon}]
\end{align}
and $\rho_t^{[k]}(\psi)=\rho_t^{[k],0}(\psi)$ for $k \in \mathbb{N}$.

Therefore, if we set
\begin{align*}
&a(t)=\begin{pmatrix}
\partial \alpha(X_t^{[0]})&O\\
O&O
\end{pmatrix},~~b(t)=\begin{pmatrix}
\beta(X_t^{[0]})\\ I_{m_1}
\end{pmatrix},\\
&c_0(t)=h(X_t^{[0]}),~~c(t)=\begin{pmatrix}
\partial h(X_t^{[0]})\epsilon&O
\end{pmatrix},\\
&H_t^{[1]}=U_t^{[1]}=0,~~H_t^{[k]}=\rho_t^{[k]}(h)~~(2\leq k\leq n-1),~~r_t^{\epsilon,n}=\rho_t^{[n],\epsilon}(h),
\end{align*}
then it follows from Proposition \ref{prop-expansion-of-SDE} and (\ref{eq-assumption-h}) that Assumptions \ref{assumption-A0}--\ref{assumption-A3} are satisfied.

Moreover, for $\phi \in C^\infty(\mathbb{R}^{d_1})$, we consider
\begin{align*}
U_t=\phi(X_t^\epsilon)=\sum_{i=0}^{n-1} \rho_t^{[i]}(\phi)\epsilon^i+\epsilon^{n}\int_0^1 \rho_t^{[n],\epsilon u}(\phi)u^{n-1}du,
\end{align*}
and define
\begin{align*}
U_t^{[i]}=\rho_t^{[i]}(\phi)~(0\leq i\leq n-1),~~u_t^{\epsilon,n}=\epsilon^{n}\int_0^1 \rho_t^{[n],\epsilon u}(\phi)u^{n-1}du.
\end{align*}
Then Assumption \ref{assumption-A6} is satisfied. Consequently, the problem of expanding the conditional expectation $E[X_t^\epsilon|\mathcal{Y}_t^\epsilon]$ is reduced to our framework.

\begin{remark}
 (1) At this stage, it may not be immediately clear why we introduce $\tilde{X}_t=(X_t^{[1]},V_t)^\top$ as the hidden system. However, the fact that $H_t^{[i]}$ can be represented as polynomial functionals of the system $\{\tilde{X}_s\}_{0\leq s\leq t}$ is crucial at the calculation stage. 
 
 For example, due to Proposition \ref{prop-expansion-of-SDE}, $X^{[2]}$ satisfies the equation
  \begin{align*}
  dX_t^{[2]}=&\partial\alpha(X_t^{[0]})[X_t^{[2]}]dt+\partial^2\alpha(X_t^{[0]})[(X_t^{[1]})^{\otimes 2}]dt+2\partial\beta(X_t^{[0]})[X_t^{[1]}]dV_t,
  \end{align*}
  and thus we have
  \begin{align*}
  X_t^{[2]}=\int_0^t \exp\left( \int_s^t \partial\alpha(X_r^{[0]})dr  \right)\left\{ \partial^2\alpha(X_s^{[0]})[(X_s^{[1]})^{\otimes 2}]ds+2\partial\beta(X_s^{[0]})[X_s^{[1]}]dV_s \right\}.
  \end{align*}
  Hence, $X_t^{[2]}$ is a polynomial functional of $\{\tilde{X}_s\}_{0\leq s\leq t}$, and the same property holds for
  \begin{align*}
  H_t^{[2]}=\frac{1}{2}\partial^2h(X_t^{[0]})[(X_t^{[1]})^{\otimes 2}]+\frac{1}{2}\partial h(X_t^{[0]})[X_t^{[2]}].
  \end{align*}
  Higher-order terms are obtained recursively in the same way.

  This structure is essential for reducing the computation of the coefficients to systems of ordinary differential equations. The details of the calculation are beyond the scope of this paper, while the methodology is discussed in the companion preprint \citet{kurisaki2025asymptotic}.\\
  (2) In the special case where $\beta$ is a constant matrix, $H^{[k]}~(k\geq 2)$ are given as polynomials of $X^{[1]}$ since $\partial \beta=0$. For this reason, it is sufficient to set $\tilde{X}_t=X_t^{[1]}$ without coupling with $V$ in this case.
\end{remark}

\section{Asymptotic Expansion of the Conditional Expectation}\label{section:asymptotic-expansion}
In this section, we derive an asymptotic expansion of the conditional expectation $E[U_t^\epsilon|\mathcal{Y}_t^\epsilon]$ under the framework introduced in Section \ref{section-framework}. For simplicity, we assume without loss of generality that $c_0=0$ by redefining $dY_t^\epsilon$ as $dY_t^\epsilon - c_0(t)dt$.

\subsection{Kallianpur-Striebel Formula and Its Expansion}\label{section:KS-formula}
We first recall the Kallianpur-Striebel formula, which provides a closed-form representation of the conditional expectation. To this end, we state a proposition concerning a Girsanov transformation.
\begin{proposition}\label{prop-girsanov}
  Let $\{\xi_t\}_{t\geq 0}$ be a $\mathbb{R}^{d_2}$-dimensional progressively measurable process independent of $W$, and assume
  \begin{align*}
    \int_0^tE\left[ |\xi_s|^2 \right]ds<\infty.
  \end{align*} 
  for every $t\geq 0$. Then, 
  \begin{align*}
    M_t=&\exp\left( \int_0^t \xi_s^\top (\sigma(s)\sigma(s)^\top)^{-1} \sigma(s) \, dW_s \right. \\
    &\quad \left. - \frac{1}{2} \int_0^t \xi_s^\top (\sigma(s)\sigma(s)^\top)^{-1} \xi_s \, ds \right).
  \end{align*}
  is a martingale.

  Furthermore, if we define a probability measure $Q$ on $\mathcal{F}_T$ by 
  \begin{align*}
    Q(A)=E[1_AM_T]~~(A\in \mathcal{F}_T)
  \end{align*}
  for a fixed $T$, then for any multi-dimensional random variable $U$ which is independent of $W$, the law of $(U,\{Y_t^\epsilon\}_{0\le t\le T})$ under $Q$ coincides with the law of $(U,\{\tilde Y_t^\epsilon\}_{0\le t\le T})$ under $P$, where
  \begin{align*}
    d\tilde{Y}_t^\epsilon=\{H_t^\epsilon+\xi_t\}dt+\sigma(t)dW_t,~~\tilde{Y}_0^\epsilon=0.
  \end{align*}
\end{proposition}
\begin{proof}
  For the proof of the first part, see Theorem 2.8 in \citep{kunita2011nonlinear}. The second part is the well-known Girsanov transformation.
\end{proof}

We now fix $T>0$, and define a probability measure $Q^\epsilon$ by
\begin{align*}
  Q^\epsilon(A) = E\left[ 1_A (Z_T^\epsilon)^{-1} \right]
\end{align*}
for \( A \in \mathcal{F} \), where
\begin{align}
  \begin{split}
    Z_t^\epsilon &= \exp\left( \int_0^t (H_s^\epsilon)^\top (\sigma(s)\sigma(s)^\top)^{-1} \sigma(s) \, dW_s \right. \\
    &\quad \left. + \frac{1}{2} \int_0^t (H_s^\epsilon)^\top (\sigma(s)\sigma(s)^\top)^{-1} H_s^\epsilon \, ds \right) \\    
    &= \exp\left( \int_0^t (H_s^\epsilon)^\top (\sigma(s)\sigma(s)^\top)^{-1} \, dY_s^\epsilon \right. \\
    &\quad \left. - \frac{1}{2} \int_0^t (H_s^\epsilon)^\top (\sigma(s)\sigma(s)^\top)^{-1} H_s^\epsilon \, ds \right).
  \end{split}\label{eq-exp-Z}   
\end{align}
By Proposition \ref{prop-girsanov} with $\xi_t=-H_t^\epsilon$, the probability measure $Q^\epsilon$ is well defined, and the following result holds.

\begin{proposition}\label{prop-girsanov-2}
  Under the new measure \( Q^\epsilon \), the process
  \[
  \int_0^t (\sigma(s)\sigma(s)^\top)^{-\frac{1}{2}} dY_s^\epsilon
  \]
  is a Brownian motion. Furthermore, if a random variable $U$ is independent of $W$ under $P$, then it is also independent of $\{Y_t^\epsilon\}_{0\le t\le T}$ under $Q^\epsilon$, and its law is the same under both $P$ and $Q^\epsilon$.
\end{proposition}

Moreover, the following formula holds.

\begin{proposition}[Kallianpur-Striebel formula]\label{prop-Kallianpur-striebel}
  If a random variable $U$ is independent of $W$ under $P$ and \( E[|U|] < \infty \), then it holds that
\begin{align}
  \label{eq-Kallianpur-Striebel-epsilon}
  E[U|\mathcal{Y}_t^\epsilon] =
  \frac{\displaystyle E_{Q^\epsilon}\left[ U Z_t^\epsilon \middle| \mathcal{Y}_t^\epsilon \right]}
       {\displaystyle E_{Q^\epsilon}\left[ Z_t^\epsilon \middle| \mathcal{Y}_t^\epsilon \right]}
  \quad (0 \leq t \leq T).
\end{align}
\end{proposition}

\begin{proof}
  See \citet{bain2009fundamentals}.
\end{proof}

The goal of this section is to derive an asymptotic expansion of (\ref{eq-Kallianpur-Striebel-epsilon}) by expanding both the numerator and the denominator with respect to $\epsilon$. We first expand $Z_t^\epsilon$ as a power series in $\epsilon$ and obtain the following formula.
\begin{proposition}\label{prop-expansion-Z}
  For any $0\leq t\leq T$, it holds
  \begin{align}
    U_t^\epsilon Z_t^\epsilon
    =&Z_t^0 U_t^\epsilon\left( 1+\sum_{i=1}^{n-1}I_t^{i,\epsilon}+R_t^\epsilon \right),\label{eq-expansion-integrand}
  \end{align}
  where
  \begin{align}
    r_s^{1,\epsilon}=&H_s^\epsilon - c(s)X_s\nonumber\\
    \begin{split}
      I_t^{i,\epsilon}=&\int_0^t\int_0^{t_{i}}\cdots \int_{0}^{t_2} (r_{t_1}^{1,\epsilon})^\top (\sigma(t_1)\sigma(t_1)^\top)^{-1} (dY_{t_1}^\epsilon-c(t_1)X_{t_1}dt_1)\\
  &\times \cdots \times (r_{t_i}^{1,\epsilon})^\top (\sigma(t_i)\sigma(t_i)^\top)^{-1} (dY_{t_i}^\epsilon-c(t_i)X_{t_i}dt_i),
    \end{split}\label{def-I-hat}\\
    \begin{split}
      R_t^\epsilon=&\int_0^t\int_0^{t_{n}}\cdots \int_{0}^{t_2} K_{t_1}^\epsilon(r_{t_1}^{1,\epsilon})^\top (\sigma(t_1)\sigma(t_1)^\top)^{-1} (dY_{t_1}^\epsilon-c(t_1)X_{t_1}dt_1)\\
  &\times \cdots \times (r_{t_n}^{1,\epsilon})^\top (\sigma(t_n)\sigma(t_n)^\top)^{-1} (dY_{t_n}^\epsilon-c(t_n)X_{t_n}dt_n),
    \end{split} \label{def-R-hat}\\
    \begin{split}
      K_t^\epsilon=&\exp\left( \int_0^t(r_s^{1,\epsilon})^\top (\sigma(s)\sigma(s)^\top)^{-1} (dY_s^\epsilon-c(s)X_sds)\right.\\
  &\left.-\frac{1}{2}\int_0^t(r_s^{1,\epsilon})^\top(\sigma(s)\sigma(s)^\top)^{-1} r_s^{1,\epsilon} ds\right),
    \end{split}  \label{def-K-hat}
  \end{align}
  and
  \begin{align}
    \begin{split}
      Z_t^0=&\exp\left( \int_0^t(c(s)X_s)^\top (\sigma(s)\sigma(s)^\top)^{-1} dY_s^\epsilon\right.\\
  &\left.-\frac{1}{2}\int_0^t(c(s)X_s)^\top(\sigma(s)\sigma(s)^\top)^{-1} c(s)X_s ds\right).
    \end{split}\label{def-Z0}    
  \end{align}  
\end{proposition}
\begin{proof}
  Note that we can write $Z_t^\epsilon=Z_t^0K_t^\epsilon$. Then by repeatedly employing It\^o's formula, we have
\begin{align}
  K_t^\epsilon&=1+\int_0^t K_{s}^\epsilon (r_{s}^{1,\epsilon})^\top (\sigma(s)\sigma(s)^\top)^{-1} (dY_{s}^\epsilon-c(s)X_sds)\nonumber\\
  &=\cdots\cdots\nonumber\\
  \begin{split}
    &=1+\sum_{i=1}^{n-1}\int_0^t\int_0^{t_{i}}\cdots \int_{0}^{t_2} (r_{t_1}^{1,\epsilon})^\top (\sigma(t_1)\sigma(t_1)^\top)^{-1} (dY_{t_1}^\epsilon-c(t_1)X_{t_1}dt_1)\\
  &\times \cdots \times (r_{t_i}^{1,\epsilon})^\top (\sigma(t_i)\sigma(t_i)^\top)^{-1} (dY_{t_i}^\epsilon-c(t_i)X_{t_i}dt_i)\\
  &+\int_0^t\int_0^{t_{n}}\cdots \int_{0}^{t_2} K_{t_1}^\epsilon(r_{t_1}^{1,\epsilon})^\top (\sigma(t_1)\sigma(t_1)^\top)^{-1} (dY_{t_1}^\epsilon-c(t_1)X_{t_1}dt_1)\\
  &\times \cdots \times (r_{t_n}^{1,\epsilon})^\top (\sigma(t_n)\sigma(t_n)^\top)^{-1} (dY_{t_n}^\epsilon-c(t_n)X_{t_n}dt_n),
  \end{split}\label{eq-expansion-K}
\end{align}
which yields the conclusion.
\end{proof}

Using this representation, we consider the expansion of $E_{Q^\epsilon}[U_t^\epsilon Z_t^\epsilon|\mathcal{Y}_t^\epsilon]$. To this end, we introduce a new probability measure $\hat{P}^\epsilon$ on $\mathcal{F}_T$ by
\begin{align*}
  \hat{P}^\epsilon(A)=E_{Q^\epsilon}[1_A Z_T^0] \quad (A\in \mathcal{F}_T).
\end{align*}
Recalling Proposition \ref{prop-girsanov-2} and applying Proposition \ref{prop-girsanov} again with $\xi_t=c(t)X_t$, we obtain the following result.

\begin{proposition}\label{prop-girsanov-3}
  For any multidimensional random variable $U$ that is independent of $W$ under $P$, the law of $(U, \{Y_t^\epsilon\}_{0\leq t\leq T})$ under $\hat{P}^\epsilon$ is equivalent to the law of $(U, \{\hat{Y}_t^\epsilon\}_{0\leq t\leq T})$ under $P$, where
\begin{align*}
  d\hat{Y}_t^\epsilon = c(t)X_t\,dt + \sigma(t)dW_t, \qquad \hat{Y}_0^\epsilon = 0.
\end{align*}
\end{proposition}

Using this result, if we define
\begin{align*}
    \hat{I}_t^{i,\epsilon}=&\int_0^t\int_0^{t_{i}}\cdots \int_{0}^{t_2} (r_{t_1}^{1,\epsilon})^\top (\sigma(t_1)\sigma(t_1)^\top)^{-1} \sigma(t_1)dW_{t_1}\\
  &\times \cdots \times (r_{t_i}^{1,\epsilon})^\top (\sigma(t_i)\sigma(t_i)^\top)^{-1} \sigma(t_i)dW_{t_i},\\
    \hat{R}_t^\epsilon=&\int_0^t\int_0^{t_{n}}\cdots \int_{0}^{t_2} \hat{K}_{t_1}^\epsilon(r_{t_1}^{1,\epsilon})^\top (\sigma(t_1)\sigma(t_1)^\top)^{-1} \sigma(t_1)dW_{t_1}\\
  &\times \cdots \times (r_{t_n}^{1,\epsilon})^\top (\sigma(t_n)\sigma(t_n)^\top)^{-1} \sigma(t_n)dW_{t_n},\\
  \hat{K}_t^\epsilon=&\exp\left( \int_0^t(r_s^{1,\epsilon})^\top (\sigma(s)\sigma(s)^\top)^{-1}\sigma(s)dW_s\right.\\
  &\left.-\frac{1}{2}\int_0^t(r_s^{1,\epsilon})^\top(\sigma(s)\sigma(s)^\top)^{-1} r_s^{1,\epsilon} ds\right), 
  \end{align*}
  Then it holds that
\begin{align*}
  E_{Q^\epsilon}\!\left[ Z_t^0 |U_t^\epsilon I_t^{i,\epsilon}| \right]
  =E_{\hat{P}^\epsilon}\!\left[ |U_t^\epsilon I_t^{i,\epsilon}| \right]
  =E\!\left[ |U_t^\epsilon \hat{I}_t^{i,\epsilon}| \right].
\end{align*}
Since it immediately follows from Assumptions \ref{assumption-A2} and \ref{assumption-A3} that this expectation is finite, we can take the conditional expectation of the first $n$ terms of (\ref{eq-expansion-integrand}).

On the other hand, the evaluation of the remainder term is difficult, since $E[(K_t^\epsilon)^p]<\infty$ for $p>1$ does not hold without imposing strong regularity assumptions on $r_t^{1,\epsilon}$, and therefore we cannot apply H\"older's inequality.

\subsection{Estimation of the Remainder Term}\label{section:estimation-residue}
In view of the above difficulty, we instead aim to establish probabilistic control of the remainder term. Although this is weaker than $L^p$ estimates, it can be obtained within our general framework without imposing additional assumptions.

In the sequel, we first consider the expansion under the measure \( Q^\epsilon \) before transitioning to the problem under the original measure \( P \). For this purpose, we define the following notation:

\begin{definition}
  Let
\begin{align*}
  \eta_t^\epsilon = O_{\nu^\epsilon}^{T}(1)
\end{align*}
if a family of random variables $\{\eta_t^\epsilon\}_{0\leq \epsilon<1,t\geq 0}$ and family of probability measures $\{\nu_\epsilon\}_{0\leq \epsilon <1}$ on $\mathcal{F}$ satisfies
\begin{align*}
  \lim_{K\to \infty}\sup_{0\leq t\leq T}\sup_{0<\epsilon <r}\nu^\epsilon(|\eta_t^\epsilon|>K)=0
\end{align*}
for some $r>0$.
Furthermore, let us write $\eta^\epsilon=O_{\nu^\epsilon}^{T}(\epsilon^i)$ if
\begin{align*}
  \frac{\eta^\epsilon}{\epsilon^i}=O_{\nu^\epsilon}^{T}(1)
\end{align*}
for $i=0,1,2,\cdots$.
\end{definition}
\begin{remark}
  If $\eta_t^\epsilon=O_P^T(\epsilon^i)$ and $\theta_t^\epsilon=O_P^T(\epsilon^j)$, then it follows that $|\eta_t^\epsilon||\theta_t^\epsilon|=O_P^T(\epsilon^{i+j})$. In fact, for any $\delta>0$, if we take $K_1,K_2>0$ and $r>0$ such that  
  \begin{align*}
    P\left( \frac{|\eta_t^\epsilon|}{\epsilon^i}\geq K_1 \right)<\delta,~~
    P\left( \frac{|\theta_t^\epsilon|}{\epsilon^j}\geq K_2 \right)<\delta
  \end{align*}
  for any $0<\epsilon<r$ and $0\leq t \leq T$, then it holds that 
  \begin{align*}
    P\left( \frac{|\eta_t^\epsilon||\theta_t^\epsilon|}{\epsilon^{i+j}}< K_1K_2 \right)\geq 1-2\delta,
  \end{align*}
  which implies $|\eta_t^\epsilon||\theta_t^\epsilon|=O_P^T(\epsilon^{i+j})$. This result will be used frequently without any remarks in the following part. 
\end{remark}

Furthermore, the following lemma holds.
\begin{lemma}\label{lemma-equivalence-OP-OQ}
  Let $i=0,1,2,\cdots$. If a family of random variables $\{\eta_t^\epsilon\}_{0<\epsilon<1,t\geq 0}$ satisfies
  \begin{align*}
    \eta_t^\epsilon = O_{Q^\epsilon}^{T}(\epsilon^i),
  \end{align*}
  then it holds
  \begin{align*}
    \eta_t^\epsilon = O_{P}^{T}(\epsilon^i).
  \end{align*}
\end{lemma}
\begin{proof}
  It suffices to consider the case \( i = 0 \). By Assumption \ref{assumption-A1} and the Kolmogorov continuity theorem, we can take a version of \( Z_T^\epsilon \) that is continuous with respect to \( \epsilon \). In particular, we have
\begin{align*}
  \sup_{0 < \epsilon < 1} |Z_T^\epsilon(\omega)| < \infty
\end{align*}
for every \( \omega \in \Omega \). Therefore, for any \( \delta > 0 \), there exists a constant \( C \) such that
\begin{align*}
  \sup_{0 < \epsilon < 1} P(|Z_T^\epsilon| \geq C) < \frac{\delta}{2}.
\end{align*}

  On the other hand, it follows from $\eta_t^\epsilon = O_{Q^\epsilon}^{T}(1)$ that there exists a constant $K>0$ such that 
  \begin{align*}
    \sup_{0\leq t\leq T}\sup_{0< \epsilon <r}Q^\epsilon (|\eta_t^\epsilon|\geq K)<\frac{\delta}{2C}
  \end{align*}
  for some $r>0$. Then it holds for any $0\leq t\leq T$ and $0< \epsilon <r$ that 
  \begin{align*}
    P(|\eta_t^\epsilon|>K)=&E_{Q^\epsilon}[1_{\{|\eta_t^\epsilon|>K\}}Z_T^\epsilon]\\
    =&E_{Q^\epsilon}[1_{\{|\eta_t^\epsilon|>K\}}Z_T^\epsilon 1_{\{Z_T^\epsilon < C\}}]
    +E_{Q^\epsilon}[1_{\{|\eta_t^\epsilon|>K\}}Z_T^\epsilon 1_{\{Z_T^\epsilon \geq C\}}]\\
    \leq &CQ^\epsilon(|\eta_t^\epsilon|>K)+P(|\eta_t^\epsilon|>K,Z_T^\epsilon \geq C)\\
    \leq &C\times \frac{\delta}{2C}+\frac{\delta}{2}=\delta,
  \end{align*}
  which implies the conclusion.
\end{proof}

We now have the following result.
\begin{theorem}\label{theorem-asymptotic-expansion-1}
  For any $0\leq t\leq T$ and $n \in \mathbb{N}$, it holds that 
  \begin{align}
    \begin{split}
      &E_{Q^\epsilon}[U_t^\epsilon Z_t^\epsilon|\mathcal{Y}_t^\epsilon]\\
    =&E_{Q^\epsilon}[U_{t}^{n-1,\epsilon} Z_t^0|\mathcal{Y}_t^\epsilon]+\sum_{i=1}^{n-1}E_{Q^\epsilon}\left[U_{t}^{n-i-1,\epsilon} Z_t^0J_t^{i,\epsilon}\middle|\mathcal{Y}_t^\epsilon\right]+O_{P}^T(\epsilon^n),
    \end{split}\label{eq-expansion-numerator}
  \end{align}
  where 
  \begin{align}
    \label{eq4-24}&H_{t}^{i,\epsilon}=\sum_{j=1}^i H_t^{[j]}\epsilon^j,~~U_{t}^{i,\epsilon}=\sum_{j=0}^i U_t^{[j]}\epsilon^j\\
    \begin{split}
      &J_t^{i,\epsilon}=\int_0^t\int_0^{t_{i}}\cdots \int_{0}^{t_2} (H_{t_1}^{n-1,\epsilon})^\top (\sigma(t_1)\sigma(t_1)^\top)^{-1} (dY_{t_1}^\epsilon-c(t_1)X_{t_1}dt_1)\\
    &\times \cdots \times (H_{t_i}^{n-1,\epsilon})^\top (\sigma(t_i)\sigma(t_i)^\top)^{-1} (dY_{t_i}^\epsilon-c(t_i)X_{t_i}dt_i),
    \end{split}
  \end{align}
  for $i=0,1,2,\cdots$.
\end{theorem}
\begin{proof}
In this proof, \( C \) denotes a generic positive constant. In particular, if we write \( C_q \), it may depend on some factor \( q \). Also, we assume \( 0 \leq t \leq T \) throughout this proof.

We first define
  \begin{align*}
    A_t^\epsilon = \left\{ \sup_{0 \leq s \leq t} |H_s^\epsilon - H_s^0| \geq 1 \right\} = \left\{ \sup_{0 \leq s \leq t} |r_s^{1,\epsilon}| \geq 1 \right\}.
  \end{align*}
 Then we have
  \begin{align}
    &E_{Q^\epsilon}\left[ U_t^\epsilon Z_t^\epsilon \middle| \mathcal{Y}_t^\epsilon \right] \nonumber \\
    \label{eq4-10}
    =& \frac{E_{Q^\epsilon}\left[ U_t^\epsilon Z_t^\epsilon 1_{A_t^\epsilon} \middle| \mathcal{Y}_t^\epsilon \right]}{E_{Q^\epsilon}[Z_t^\epsilon | \mathcal{Y}_t^\epsilon]} E_{Q^\epsilon}[Z_t^\epsilon | \mathcal{Y}_t^\epsilon] 
    + E_{Q^\epsilon}\left[ U_t^\epsilon Z_t^\epsilon 1_{(A_t^\epsilon)^c} \middle| \mathcal{Y}_t^\epsilon \right].
  \end{align}
  
  For the first term on the right-hand side, (\ref{eq-Kallianpur-Striebel-epsilon}) and Assumption \ref{assumption-A3} yield
  \begin{align}
    E\left[ \left| \frac{E_{Q^\epsilon}\left[ U_t^\epsilon Z_t^\epsilon 1_{A_t^\epsilon} \middle| \mathcal{Y}_t^\epsilon \right]}{E_{Q^\epsilon}[Z_t^\epsilon | \mathcal{Y}_t^\epsilon]} \right| \right]
    &= E\left[ \left| E\left[ U_t^\epsilon 1_{A_t^\epsilon} \middle| \mathcal{Y}_t^\epsilon \right] \right| \right] \nonumber \\
    &\leq E\left[ \left| U_t^\epsilon 1_{A_t^\epsilon} \right| \right] \nonumber \\
    &\leq E\left[ |U_t^\epsilon|^p \right]^\frac{1}{p} P\left( \sup_{0 \leq s \leq t} |r_s^{1,\epsilon}| \geq 1 \right)^\frac{1}{q} \nonumber \\
    &\leq E\left[ |U_t^\epsilon|^p \right]^\frac{1}{p} E\left[ \sup_{0 \leq s \leq t} |r_s^{1,\epsilon}|^{qn} \right]^\frac{1}{q} \nonumber \\
    \label{eq4-12}
    &\leq C_T \epsilon^n,
  \end{align}
  for sufficiently large $p$ and $\frac{1}{p}+\frac{1}{q}=1$, which implies
  \begin{align*}
    \frac{E_{Q^\epsilon}\left[ U_t^\epsilon Z_t^\epsilon 1_{A_t^\epsilon} \middle| \mathcal{Y}_t^\epsilon \right]}{E_{Q^\epsilon}[Z_t^\epsilon | \mathcal{Y}_t^\epsilon]} = O_P^T(\epsilon^n).
  \end{align*}
  Furthermore, it follows from the definition of \( Q^\epsilon \) that
  \begin{align*}
    E_{Q^\epsilon}[E_{Q^\epsilon}[Z_t^\epsilon | \mathcal{Y}_t^\epsilon]] = E_{Q^\epsilon}[Z_t^\epsilon] = 1.
  \end{align*}
  Therefore, it holds that \( E_{Q^\epsilon}[Z_t^\epsilon \mid \mathcal{Y}_t^\epsilon] = O_{Q^\epsilon}^T(1) \). By Lemma \ref{lemma-equivalence-OP-OQ}, this implies that \( E_{Q^\epsilon}[Z_t^\epsilon \mid \mathcal{Y}_t^\epsilon] = O_P^T(1) \). Combining these results, we obtain
  \begin{align}
    \label{eq4-11}
    \frac{E_{Q^\epsilon}\left[ U_t^\epsilon Z_t^\epsilon 1_{A_t^\epsilon} \middle| \mathcal{Y}_t^\epsilon \right]}{E_{Q^\epsilon}[Z_t^\epsilon | \mathcal{Y}_t^\epsilon]} E_{Q^\epsilon}[Z_t^\epsilon | \mathcal{Y}_t^\epsilon] = O_P^T(\epsilon^n).
  \end{align}
  Thus, it follows from (\ref{eq4-10}) that
\begin{align}
  &E_{Q^\epsilon}\left[ U_t^\epsilon Z_t^\epsilon \middle| \mathcal{Y}_t^\epsilon \right] 
  =E_{Q^\epsilon}\left[ U_t^\epsilon Z_t^\epsilon 1_{(A_t^\epsilon)^c} \middle| \mathcal{Y}_t^\epsilon \right]+O_P^T(\epsilon^n).
  \label{eq-exp-U-Z-A}
\end{align}
By Proposition \ref{prop-expansion-Z}, the dominant term on the right-hand side can be expanded as
  \begin{align}
    &E_{Q^\epsilon}\left[ U_t^\epsilon Z_t^\epsilon 1_{(A_t^\epsilon)^c} \middle| \mathcal{Y}_t^\epsilon \right]\nonumber\\
    \begin{split}
      =&E_{Q^\epsilon}\left[ U_t^\epsilon Z_t^0 1_{(A_t^\epsilon)^c} \middle| \mathcal{Y}_t^\epsilon \right]
    +\sum_{i=1}^{n-1}E_{Q^\epsilon}\left[ U_t^\epsilon I_t^{i,\epsilon}Z_t^0 1_{(A_t^\epsilon)^c} \middle| \mathcal{Y}_t^\epsilon \right]\\
    &+E_{Q^\epsilon}\left[ U_t^\epsilon R_t^{\epsilon}Z_t^0 1_{(A_t^\epsilon)^c} \middle| \mathcal{Y}_t^\epsilon \right].
    \end{split}\label{eq-decomposition-UZA}    
  \end{align}
  For the first $n$ terms on the right-hand side, it follows from Proposition \ref{prop-girsanov-3} together with Assumptions \ref{assumption-A2}, \ref{assumption-A3}, and \ref{assumption-A6} that
  \begin{align*}
    &E_{Q^\epsilon}\left[ \left|E_{Q^\epsilon}\left[ U_t^\epsilon I_t^{i,\epsilon}Z_t^0 1_{(A_t^\epsilon)^c} \middle| \mathcal{Y}_t^\epsilon \right]-E_{Q^\epsilon}\left[ U_t^{n-i-1,\epsilon} J_t^{i,\epsilon}Z_t^0 \middle| \mathcal{Y}_t^\epsilon \right] \right| \right]\\
    \leq &E_{Q^\epsilon}\left[ Z_0|U_t^\epsilon I_t^{i,\epsilon}1_{(A_t^\epsilon)^c}-U_t^{n-i-1,\epsilon} J_t^{i,\epsilon}| \right]\\
    =&E_{\tilde{P}^\epsilon}\left[ |U_t^\epsilon I_t^{i,\epsilon}1_{(A_t^\epsilon)^c}-U_t^{n-i-1,\epsilon} J_t^{i,\epsilon}| \right]\\
    =&E\left[ |U_t^\epsilon \hat{I}_t^{i,\epsilon}1_{(A_t^\epsilon)^c}-U_t^{n-i-1,\epsilon} \hat{J}_t^{i,\epsilon}| \right]\\
    \leq &C_T\epsilon^n,
  \end{align*}
  where 
  \begin{align*}
    \hat{J}_t^{i,\epsilon}=&\int_0^t\int_0^{t_{i}}\cdots \int_{0}^{t_2} (H_{t_1}^{n-1,\epsilon})^\top (\sigma(t_1)\sigma(t_1)^\top)^{-1} \sigma(t_1)dW_{t_1}\\
    &\times \cdots \times (H_{t_i}^{n-1,\epsilon})^\top (\sigma(t_i)\sigma(t_i)^\top)^{-1} \sigma(t_i)dW_{t_i}
  \end{align*}
  and $\hat{I}_t^{i,\epsilon}$ is defined in (\ref{def-I-hat}). Therefore, it follows from Lemma \ref{lemma-equivalence-OP-OQ} that
  \begin{align}
    E_{Q^\epsilon}\left[ U_t^\epsilon I_t^{i,\epsilon}Z_t^0 1_{(A_t^\epsilon)^c} \middle| \mathcal{Y}_t^\epsilon \right]=E_{Q^\epsilon}\left[ U_t^{n-i-1,\epsilon} J_t^{i,\epsilon}Z_t^0 \middle| \mathcal{Y}_t^\epsilon \right]+O_P^T(\epsilon^n).\label{eq-estimate-difference}
  \end{align}
  For the final term of (\ref{eq-decomposition-UZA}), noting that for every $p\geq 1$, it holds
\begin{align}
  &E[|1_{(A_t^\epsilon)^c}\hat{K}_{s}|^p]\nonumber\\
  =&E\Biggl[1_{(A_t^\epsilon)^c}\exp\left( p\int_0^s(r_u^{1,\epsilon})^\top (\sigma(u)\sigma(u)^\top)^{-1} dW_u\right.\nonumber\\
  &\left.\qquad\qquad \qquad\qquad-\frac{p}{2}\int_0^s(r_u^{1,\epsilon})^\top(\sigma(u)\sigma(u)^\top)^{-1} r_u^{1,\epsilon} du\right)\Biggr]\nonumber\\
  = &E_{Q^\epsilon}\Biggl[\exp\left( \int_0^s(2p r_u^{1,\epsilon})^\top (\sigma(u)\sigma(u)^\top)^{-1} dW_u\right.\nonumber\\
  &\left.\qquad\qquad \qquad\qquad-p^2\int_0^s(r_u^{1,\epsilon})^\top(\sigma(u)\sigma(u)^\top)^{-1} r_u^{1,\epsilon} ds\right)\nonumber\\
  &\times 1_{(A_t^\epsilon)^c}\exp\left( \frac{2p^2-p}{2}\int_0^s(r_u^{1,\epsilon})^\top(\sigma(u)\sigma(u)^\top)^{-1} r_u^{1,\epsilon} du\right) \Biggr]\nonumber\\
  \leq &E_{Q^\epsilon}\Biggl[\exp\left( \int_0^s(2p r_u^{1,\epsilon})^\top (\sigma(u)\sigma(u)^\top)^{-1} dW_u\right.\nonumber\\
  &\left.\qquad\qquad \qquad\qquad -\frac{1}{2}\int_0^s(2pr_u^{1,\epsilon})^\top(\sigma(u)\sigma(u)^\top)^{-1} 2pr_u^{1,\epsilon} du\right)\Biggr]^\frac{1}{2}\nonumber\\
  &\times E_{Q^\epsilon}\left[ 1_{(A_t^\epsilon)^c}\exp\left( \frac{2p^2-p}{2}\int_0^s(r_u^{1,\epsilon})^\top(\sigma(u)\sigma(u)^\top)^{-1} r_u^{1,\epsilon} du\right) \right]^\frac{1}{2}\nonumber\\
  \leq &C_{p,T}\label{eq-estimate-K}
\end{align}
and that $1_{(A_t^\epsilon)^c}$ is independent of $W$, we obtain 
\begin{align}
  &E_{Q^\epsilon}\left[ \left|E_{Q^\epsilon}\left[ U_t^\epsilon R_t^{\epsilon}Z_t^0 1_{(A_t^\epsilon)^c} \middle| \mathcal{Y}_t^\epsilon \right]\right| \right]\nonumber\\
  \leq &E\left[ |U_t^\epsilon \hat{R}_t^{\epsilon}| 1_{(A_t^\epsilon)^c} \right]\nonumber\\
  =&E[|U_t^\epsilon|^p]^\frac{1}{p}E\Biggl[\left|1_{(A_t^\epsilon)^c}\int_0^t\int_0^{t_{n}}\cdots \int_{0}^{t_2} \hat{K}_{t_1}^\epsilon(r_{t_1}^{1,\epsilon})^\top (\sigma(t_1)\sigma(t_1)^\top)^{-1} \sigma(t_1)dW_{t_1}\right.\nonumber\\
  &\left.\times \cdots \times (r_{t_n}^{1,\epsilon})^\top (\sigma(t_n)\sigma(t_n)^\top)^{-1} \sigma(t_n)dW_{t_n}\right|^q\Biggr]^\frac{1}{q} \nonumber\\
  \label{eq4-15}\begin{split}
    \leq &C_T \int_0^t\int_0^{t_{n}}\cdots \int_{0}^{t_2} E\Bigl[1_{(A_t^\epsilon)^c}(\hat{K}_{t_1}^\epsilon)^q |r_{t_1}^{1,\epsilon}|^q\cdots |r_{t_n}^{1,\epsilon}|^q\Bigr]dt_1\cdots dt_n.
  \end{split}  
\end{align}
By H\"older's inequality, Assumption \ref{assumption-A3}, and (\ref{eq-estimate-K}), it follows that the term in (\ref{eq4-15}) is \( O(\epsilon^n) \). Thus, by Lemma \ref{lemma-equivalence-OP-OQ}, we conclude that
\begin{align}
  E_{Q^\epsilon}\left[ U_t^\epsilon R_t^{\epsilon} Z_t^0 1_{(A_t^\epsilon)^c} \middle| \mathcal{Y}_t^\epsilon \right]
  = O_P^T(\epsilon^n). 
  \label{eq-estimate-remainder}
\end{align}

Combining (\ref{eq4-11}), (\ref{eq-decomposition-UZA}), (\ref{eq-estimate-difference}), and (\ref{eq-estimate-remainder}), we obtain the desired result.
\end{proof}

By Theorem \ref{theorem-asymptotic-expansion-1}, we can derive the asymptotic expansion of the quotient (\ref{eq-Kallianpur-Striebel-epsilon}). Let us write 
\begin{align}
  \label{eq-def-tilde-E-epsilon}
  \tilde{E}_t^\epsilon[U|\mathcal{Y}_t^\epsilon]
  =
  \frac{E_{Q^\epsilon}[UZ_t^0|\mathcal{Y}_t^\epsilon]}
       {E_{Q^\epsilon}[Z_t^0|\mathcal{Y}_t^\epsilon]}
\end{align}
for $U \in L^1(\Omega,\mathcal{X}_t,P)$. 
Note that the right-hand side is well defined since $Z_t^0>0$ almost surely.

To control the denominator in (\ref{eq-def-tilde-E-epsilon}), we establish the following lower bound. 
Lemmas \ref{lemma-boundedness-phi} and \ref{prop-expectation-zeta} in Section \ref{section-appendix} yield the following result.
\begin{proposition}\label{prop-lower-bound}
  We can take a version of $E_{Q^\epsilon}[Z_t^0|\mathcal{Y}_t^\epsilon]$ such that almost surely
  \begin{align*}
    \inf_{0\leq t\leq T}E_{Q^\epsilon}[Z_t^0|\mathcal{Y}_t^\epsilon]>0.
  \end{align*}
\end{proposition}

Using this result, the quotient (\ref{eq-Kallianpur-Striebel-epsilon}) is expanded as follows.

\begin{theorem}\label{theorem-asymptotic-expansion-2}
  Under Assumptions \ref{assumption-A0}--\ref{assumption-A6}, it holds that
  \begin{align*}
    &E[U_t^\epsilon |\mathcal{Y}_t^\epsilon]\\
    =&\frac{\displaystyle \tilde{E}_t^\epsilon[U_{t}^{n-1,\epsilon} |\mathcal{Y}_t^\epsilon]+\sum_{i=1}^{n-1}\tilde{E}_t^\epsilon[U_{t}^{n-i-1,\epsilon} J_t^i|\mathcal{Y}_t^\epsilon]}{\displaystyle 1+\sum_{i=1}^{n-1}\tilde{E}_t^\epsilon[J_t^i|\mathcal{Y}_t^\epsilon]}+O_P^T(\epsilon^n)\\
    =&\Biggl(\tilde{E}_t^\epsilon[U_{t}^{n-1,\epsilon} |\mathcal{Y}_t^\epsilon]+\sum_{i=1}^{n-1}\tilde{E}_t^\epsilon[U_{t}^{n-i-1,\epsilon} J_t^i|\mathcal{Y}_t^\epsilon]\Biggr)
    \Biggl\{1+\sum_{j=1}^{n-1}(-1)^{j}\Biggl(\sum_{i=1}^{n-1}\tilde{E}_t^\epsilon[J_t^i |\mathcal{Y}_t^\epsilon]\Biggr)^{j}\Biggr\}\\
    &+O_P^T(\epsilon^n),
  \end{align*}
  for every $n \in \mathbb{N}$.
  Here, $H_{t,i}^\epsilon$ and $U_{t,i}^\epsilon$ is given in (\ref{eq4-24}), and
  \begin{align}
    \label{def-J-i}\begin{split}
      J_t^i=&\int_0^t\int_0^{t_{i}}\cdots \int_{0}^{t_2} (H_{t_1}^{n-1,\epsilon})^\top (\sigma(t_1)\sigma(t_1)^\top)^{-1} (dY_{t_1}^\epsilon-c(t_1)X_{t_1}dt_1)\\
    &\times \cdots \times (H_{t_i}^{n-1,\epsilon})^\top (\sigma(t_i)\sigma(t_i)^\top)^{-1} (dY_{t_i}^\epsilon-c(t_i)X_{t_i}dt_i).
    \end{split}    
  \end{align}
  Also, $x/0$ is interpreted as $0$ for any $x \in \mathbb{R}$.
\end{theorem}
\begin{proof}
  Let us define the following expressions:
\begin{align*}
    \Phi_{t,n-1}^\epsilon(U_t^\epsilon) &= E_{Q^\epsilon}[U_{t}^{n-1,\epsilon} Z_t^0|\mathcal{Y}_t^\epsilon] + \sum_{i=1}^{n-1} E_{Q^\epsilon}[U_{t}^{n-i-1,\epsilon} Z_t^0 J_t^i | \mathcal{Y}_t^\epsilon ],
\end{align*}
and
\begin{align*}
    \Psi_{t,n-1}^\epsilon(U_{t}^\epsilon) &= \sum_{i=1}^{n-1} \tilde{E}_t^\epsilon[U_{t}^{n-i-1,\epsilon} J_t^i | \mathcal{Y}_t^\epsilon ].
\end{align*}
Using the definitions of \( H_{t}^{i,\epsilon} \) and \( U_{t,i}^\epsilon \), we can establish
\begin{align}
    \label{eq4-21}
    \Phi_{t,n-1}^\epsilon(1) &= E_{Q^\epsilon}[Z_t^0|\mathcal{Y}_t^\epsilon] + O_P^T(\epsilon),
\end{align}
and from Theorem \ref{theorem-asymptotic-expansion-1},
\begin{align}
    \label{eq4-22}
    E_{Q^\epsilon}[Z_t^\epsilon|\mathcal{Y}_t^\epsilon] &= E_{Q^\epsilon}[Z_t^0|\mathcal{Y}_t^\epsilon] + O_P^T(\epsilon).
\end{align}
Furthermore, Proposition \ref{prop-lower-bound} shows that \( E_{Q^\epsilon}[Z_t^0|\mathcal{Y}_t^\epsilon] \) is almost surely locally bounded away from $0$ for each $\epsilon$. Since the distribution of \( E_{Q^\epsilon}[Z_t^0|\mathcal{Y}_t^\epsilon] \) under \( Q^\epsilon \) is independent of \( \epsilon \), it follows that
\begin{align}
    \label{eq4-23}
    \frac{1}{E_{Q^\epsilon}[Z_t^0|\mathcal{Y}_t^\epsilon]} = O_{Q^\epsilon}^T(1),
\end{align}
and thus \( O_P^T(1) \) by Lemma \ref{lemma-equivalence-OP-OQ}.

  Together with (\ref{eq4-21}) and (\ref{eq4-22}), for any $\delta>0$, there exists a constant $\lambda>0$ and $K>0$ such that
  \begin{align*}
    \sup_{0\leq \epsilon <r}\sup_{0\leq t\leq T}P(\Phi_{t,n-1}^\epsilon(1)> \lambda-K \epsilon,E_{Q^\epsilon}[Z_t^\epsilon|\mathcal{Y}_t^\epsilon]> \lambda-K \epsilon)<\delta.
  \end{align*}
  for some $r>0$. Therefore, let us define
  \begin{align*}
    A_t^\epsilon=\left\{ \Phi_{t,n-1}^\epsilon(1)> \lambda-K \epsilon,E_{Q^\epsilon}[Z_t^\epsilon|\mathcal{Y}_t^\epsilon]> \lambda-K \epsilon \right\}.
  \end{align*}
  Then  Theorem \ref{theorem-asymptotic-expansion-1} gives on $A_t^\epsilon$ and for $\displaystyle \epsilon<\frac{\lambda}{2K}$ that 
  \begin{align*}
    &\Biggl|E[U_t^\epsilon|\mathcal{Y}_t^\epsilon]-\frac{\tilde{E}_t^\epsilon[U_{t}^{n-1,\epsilon} |\mathcal{Y}_t^\epsilon]+\Psi_{t,n-1}^\epsilon(U_{t}^\epsilon) }{1+\Psi_{t,n-1}^\epsilon(1)}\Biggr|\\
    =&\Biggl|\frac{\displaystyle E_{Q^\epsilon}\left[ U_t^\epsilon Z_t^\epsilon\middle|\mathcal{Y}_t^\epsilon \right]}{\displaystyle E_{Q^\epsilon}\left[ Z_t^\epsilon \middle|\mathcal{Y}_t^\epsilon \right]}-\frac{\Phi_{t,{n-1}}^{\epsilon}(U_{t}^\epsilon)}{\Phi_{t,{n-1}}^{\epsilon}(1)}\Biggr|\\
    \leq &\Biggl|\frac{ E_{Q^\epsilon}\left[ U_t^\epsilon Z_t^\epsilon\middle|\mathcal{Y}_t^\epsilon \right]-\Phi_{t,{n-1}}^{\epsilon}(U_{t}^\epsilon)}{E_{Q^\epsilon}\left[ Z_t^\epsilon \middle|\mathcal{Y}_t^\epsilon \right]}\Biggr|\\
    &+\Biggl|\frac{\Phi_{t,{n-1}}^{\epsilon}(U_{t}^\epsilon) \left\{ E_{Q^\epsilon}[Z_t^\epsilon|\mathcal{Y}_t^\epsilon]-\Phi_{t,{n-1}}^{\epsilon}(1) \right\}}{E_{Q^\epsilon}\left[ Z_t^\epsilon \middle|\mathcal{Y}_t^\epsilon \right]\Phi_{t,{n-1}}^{\epsilon}(1)}\Biggr|\\
    \leq &\frac{2\left| E_{Q^\epsilon}\left[ U_t^\epsilon Z_t^\epsilon\middle|\mathcal{Y}_t^\epsilon \right]-\Phi_{t,{n-1}}^{\epsilon}(U_{t}^\epsilon) \right|}{\lambda}\\
    &+\frac{4|\Phi_{t,{n-1}}^{\epsilon}(U_{t}^\epsilon)| \left\{ E_{Q^\epsilon}[Z_t^\epsilon|\mathcal{Y}_t^\epsilon]-\Phi_{t,{n-1}}^{\epsilon}(1) \right\}}{\lambda^2}=O_P^T(\epsilon^n),
  \end{align*}
  which implies the first equality.

  Furthermore, it follows from Theorem \ref{theorem-asymptotic-expansion-1} and (\ref{eq4-21}) that 
\begin{align*} 
  &\Biggl|\frac{\tilde{E}_t^\epsilon[U_{t}^{n-1,\epsilon} |\mathcal{Y}_t^\epsilon]+\Psi_{t,n-1}^\epsilon(U_{t}^\epsilon) }{1+\Psi_{t}^\epsilon(1)}\\\
  &-\left\{ \tilde{E}_t^\epsilon[U_{t}^{n-1,\epsilon} |\mathcal{Y}_t^\epsilon]+\Psi_{t,n-1}^\epsilon(U_{t}^\epsilon) \right\}\left\{ 1-\sum_{j=1}^{n-1}(-1)^j\Psi_{t,n-1}^\epsilon(1)^j \right\}\Biggr|\\
  =&\Biggl|\frac{\tilde{E}_t^\epsilon[U_{t}^{n-1,\epsilon} |\mathcal{Y}_t^\epsilon]+\Psi_{t,n-1}^\epsilon(U_{t}^\epsilon)}{1+\Psi_{t,n-1}^\epsilon(1)}(-1)^n\Psi_{t,n-1}^\epsilon(1)^n\Biggr|\\
  = &\Biggl|\frac{\Phi_{t,{n-1}}^{\epsilon}(U_{t}^\epsilon)}{\Phi_{t,{n-1}}^{\epsilon}(1)}(-1)^n\left( \frac{\Phi_{t,n-1}^\epsilon(U_{t}^\epsilon)-E_{Q^\epsilon}[U_{t}^{n-1,\epsilon} Z_t^0|\mathcal{Y}_t^\epsilon]}{E_{Q^\epsilon}[Z_t^0|\mathcal{Y}_t^\epsilon]} \right)^n\Biggr|\\
    \leq &\frac{2|\Phi_{t,{n-1}}^{\epsilon}(U_{t}^\epsilon)|}{\lambda}\left| \frac{\Phi_{t,n-1}^\epsilon(U_{t}^\epsilon)-E_{Q^\epsilon}[U_{t}^{n-1,\epsilon} Z_t^0|\mathcal{Y}_t^\epsilon]}{E_{Q^\epsilon}[Z_t^0|\mathcal{Y}_t^\epsilon]} \right|^n=O_P^T(\epsilon^n)
\end{align*}
on $A_t^\epsilon$ and for $\displaystyle \epsilon<\frac{\lambda}{2K}$. This leads to the second equality.

\end{proof}

\section{Discussion}
In the final form, our expansion is ultimately reduced to the evaluation of expectations of $U_t^{j,\epsilon}J_t^i$ under $\tilde{E}_t^\epsilon$, where $J_t^i$ is given in (\ref{def-J-i}). Since $\tilde{E}_t^\epsilon$ is associated with the reference Gaussian model (which is obtained by setting $\epsilon$ as 0 in (\ref{eq-Y-epsilon})), the conditional law of $\{X_s\}_{0\leq s\leq t}$ given $\mathcal{Y}_t^\epsilon$ is Gaussian under $\tilde{E}_t^\epsilon$.

Furthermore, in the example of Section \ref{section-example}, $U_t^{[i]}$ and $H_t^{[i]}$ are polynomial functionals of $X$ (more precisely, of $X^{[1]}$ in that section). This means that our expansion is reduced to the computation of polynomial functionals of a Gaussian process, which is why the coefficients are analytically tractable.

In subsequent work, we will provide a recursive procedure for computing these coefficients. The idea is similar in spirit to the derivation of closed systems of ordinary differential equations in \citet{YOSHIDA200353, MASUDA200437}. The key difference, however, is that we directly expand the conditional expectation through the Kallianpur--Striebel formula, rather than relying on an expansion of the finite-dimensional joint distribution. This formulation allows us to obtain a recursive representation for the conditional expectation itself.

Finally, although the present paper focuses only on the expansion of the conditional expectation, the same framework can also be used to derive an expansion of the conditional distribution, in much the same way as in \citet{YOSHIDA200353, MASUDA200437}, through expansions of moments and the characteristic function. This point will also be addressed in a subsequent paper.

As a result, our approach provides a framework in which the filtering expectation and distribution can be approximated through a finite-dimensional recursive structure, in a way analogous to asymptotic expansions of ordinary expectations. This suggests a potentially efficient and tractable alternative to existing filtering methods and expansion-based approaches, particularly in regimes where perturbative structures can be exploited.

\appendix
\section{Estimates for $E_{Q^\epsilon}[Z_t^0|\mathcal{Y}_t^\epsilon]$}\label{section-appendix}
In this section, we derive an explicit formula for $E_{Q^\epsilon}[Z_t^0|\mathcal{Y}_t^\epsilon]$ ($Z_0$ is given in (\ref{def-Z0})) as a functional of $Y^\epsilon$. To this end, we define $\hat{Y}_t$ by
  \begin{align*}
    \hat{Y}_s^\epsilon=\int_0^s c(u)^\top(\sigma(u)\sigma(u)^\top)^{-1}dY_u^\epsilon,
  \end{align*}
and rewrite $Z_t^0$ as
\begin{align*}
  Z_t^0=&\exp\left( \int_0^tX_s^\top c(s)^\top(\sigma(s)\sigma(s)^\top)^{-1} dY_s^\epsilon\right.\nonumber\\
  &\left.-\frac{1}{2}\int_0^tX_s^\top c(s)^\top(\sigma(s)\sigma(s)^\top)^{-1} c(s)X_s ds\right)\nonumber\\
    =&\exp\left( X_t^\top \hat{Y}_t^\epsilon- \int_0^t(\hat{Y}_s^\epsilon)^\top dX_s\right.\\
  &\left.-\frac{1}{2}\int_0^tX_s^\top c(s)^\top(\sigma(s)\sigma(s)^\top)^{-1} c(s)X_s ds\right).
\end{align*}
Furthermore, let us write
\begin{align*}
    z_t(y)=&\exp\left( X_t^\top y(t)- \int_0^ty(s)^\top dX_s\right.\\
  &\left.-\frac{1}{2}\int_0^tX_s^\top c(s)^\top(\sigma(s)\sigma(s)^\top)^{-1} c(s)X_s ds\right)
\end{align*}
for a continuous function $y:[0,t]\to \mathbb{R}^{d_1}$. Then we can write
\begin{align*}
  E_{Q^\epsilon}[Z_t^0 \mid \mathcal{Y}_t^\epsilon]
= E[z_t(y)]\big|_{y=\hat{Y}_{\cdot}^\epsilon},
\end{align*}
where $\hat{Y}_{\cdot}^\epsilon$ denotes the path of $\{\hat{Y}_s^\epsilon\}_{0\le s\le t}$ (see Lemma A.1 in \citet{kurisaki2026pathwise} for a precise justification). Therefore, it suffices to compute the expectation $E[z_t(y)]$ for arbitrary $y$.

For this purpose, write $S(t)=c(t)^\top(\sigma(t)\sigma(t)^\top)^{-1} c(t)$, and introduce an \( M_{d_1}(\mathbb{R}) \)-valued differentiable function \( \phi(s; t) \) as the negative semidefinite symmetric solution of the matrix Riccati equation
\begin{align}
  \label{eq-def-phi}
  \begin{split}
    \frac{d\phi}{ds}(s; t) &= -\phi(s; t) b(s) b(s)^\top \phi(s; t) - a(s)^\top \phi(s; t) - \phi(s; t) a(s)  + S(s)
  \end{split}
\end{align}
with the boundary condition \( \phi(t; t) = 0 \). This solution is well-defined. Specifically, the Riccati equation
\begin{align}
  \begin{split}
    \frac{d\tilde{\phi}}{ds}(s; t) &= -\tilde{\phi}(s; t) b(t - s) b(t - s)^\top \tilde{\phi}(s; t) \\
  &\quad + a(t - s)^\top \tilde{\phi}(s; t) + \tilde{\phi}(s; t) a(t - s) + S(t-s)\label{eq-tilde-phi}
  \end{split}  
\end{align}
with \( \tilde{\phi}(0; t) = 0 \) has a positive-semidefinite solution \( \tilde{\phi}(s; t) \) for \( 0 \leq s \leq t \), according to Theorems 2.1 and 2.2 in \citet{potter1965matrix}. It then follows immediately that \( \phi(s; t) = -\tilde{\phi}(t - s; t) \) satisfies (\ref{eq-def-phi}), and uniqueness can be established using Gronwall's lemma.

Also, let \( \{\tilde{V}_s\} \) be a \( d_1 \)-dimensional Brownian motion on \( (\Omega, \mathcal{F}, \{\mathcal{F}_t\}, P) \) that is independent of \( \{Y_s\} \) and \( X_0 \). Let \( \{\xi_{s; t}^0\}_{0 \leq s \leq t} \) be the solution to the stochastic differential equation
\begin{align*}
  d_s\xi_{s; t} = \left\{ a(s) + b(s) b(s)^\top \phi(s; t) \right\} \xi_{s; t} \, ds + b(s) \, d\tilde{V}_s,~~\xi_{0;t}=X_0.
\end{align*}

\begin{lemma}\label{lemma-measure-change}
  Let $X.$ and $\xi.0$ be the paths of $\{X_s\}_{0\leq s\leq t}$ and $\{\xi_{s,t}\}_{0\leq s\leq t}$, and $\mu_X$ and $\mu_{\xi}$ be their distributions on $P$. 
  \begin{align*}     
    \frac{d\mu_{\xi}}{d\mu_X}(X.)=\exp\biggl( &-\frac{1}{2}\int_0^t X_s^\top S(s) X_sds\\
    &-\frac{1}{2}\int_0^t b(s)^\top \phi(s;t)b(s)ds-\frac{1}{2}X_0^\top \phi(0;t) X_0\biggr).
  \end{align*}
\end{lemma}
\begin{proof}
  Due to (7.138) in \citet{Liptser2001}, $\mu_X$ and $\mu_{\xi}$ are equivalent, and it holds that 
  \begin{align}
    \label{eq3-3}\begin{split}      
    \frac{d\mu_{\xi}}{d\mu_X}(X.)=\exp\biggl(&\int_0^t X_{s}^\top\phi(s;t) dX_s\\
    &-\frac{1}{2}\int_0^tX_s^\top \phi(s;t) (2a(s)+b(s)b(s)^\top \phi(s;t))X_s ds \biggr).
    \end{split}
  \end{align}
  On the other hand, by It\^o's formula and (\ref{eq-def-phi}), we have
  \begin{align*}
    &X_t^\top \phi(t;t) X_t-X_0^\top \phi(0;t) X_0\\
    =&2\int_0^t X_s^\top\phi(s;t)dX_s+\int_0^t X_s^\top \frac{\partial}{\partial s}\phi(s;t) X_sds + \int_0^t b(s)^\top \phi(s;t)b(s)ds\\
    =&2\int_0^t X_s^\top\phi(s;t)dX_s-\int_0^t X_s^\top \phi(s;t)b(s)b(s)^\top \phi(s;t) X_sds \\
    &-\int_0^t X_s^\top a(s)^\top \phi(s;t) X_sds-\int_0^t X_s^\top  \phi(s;t)a(s) X_sds\\
    &+\int_0^t X_s^\top S(s) X_sds
    + \int_0^t b(s)^\top \phi(s;t)b(s)ds\\
    =&2\int_0^t X_s^\top\phi(s;t)dX_s-\int_0^t X_s^\top \phi(s;t)b(s)b(s)^\top \phi(s;t) X_sds \\
    &-2\int_0^t X_s^\top  \phi(s;t)a(s) X_sds\\
    &+\int_0^t X_s^\top S(s) X_sds
    + \int_0^t b(s)^\top \phi(s;t)b(s)ds.
  \end{align*}
  Therefore, noting that $\phi(t;t)=0$, it holds
  \begin{align*}
    &\int_0^t X_s^\top\phi(s;t)dX_s\\
    =&\frac{1}{2}\int_0^t X_s^\top \phi(s;t)b(s)b(s)^\top \phi(s;t) X_sds+\int_0^t X_s^\top  \phi(s;t)a(s) X_sds\\
    &-\frac{1}{2}\int_0^t X_s^\top S(s) X_sds\\
    &-\frac{1}{2}\int_0^t b(s)^\top \phi(s;t)b(s)ds-\frac{1}{2}X_0^\top \phi(0;t) X_0.
  \end{align*}  
  Together with (\ref{eq3-3}), we obtain the desired result.  
\end{proof}

\begin{lemma}\label{lemma-boundedness-phi}
  The function $\phi(s;t)$ is locally bounded with respect to $(s,t)$.
\end{lemma}
\begin{proof}
  Let us consider $\tilde{\phi}$ defined in (\ref{eq-tilde-phi}). Then it holds for $x \in \mathbb{R}^{d_1}$
  \begin{align*}
    \frac{d}{ds}x^\top \tilde{\phi}(s;t) x\leq  &2x^\top a(t - s)^\top \tilde{\phi}(s; t)x +x^\top S(t-s) x\\
    \leq &2\|a(t-s)\|_2\|\tilde{\phi}(s;t)\|_2|x|^2+\|S(t-s)\|_2|x|^2,
  \end{align*}
  where $\|\cdot\|_2$ is the matrix 2-norm. Thus it follows that 
  \begin{align*}
    \|\tilde{\phi}(s;t)\|_2\leq \int_0^s \left\{ 2\|a(t-u)\|_2\|\tilde{\phi}(u;t)\|_2+\|S(t-u)\|_2 \right\}du.
  \end{align*}
  Hence Gronwall's lemma yields
  \begin{align*}
    \|\tilde{\phi}(s;t)\|_2\leq\int_0^s\exp\left( 2\int_s^t \|a(t-r)\|_2dr \right)\|S(t-u)\|_2du,
  \end{align*}
  and the desired result follows from Assumption \ref{assumption-A0}.  
\end{proof}

\begin{lemma}\label{prop-expectation-zeta}
  We have the expression
  \begin{align*}
      &E_Q[z_t(y)]\\
    =&\sqrt{\det(I_{d_1}-\Sigma^\frac{1}{2}\phi(0;t)\Sigma^\frac{1}{2})}\exp\left( \frac{1}{2}\int_0^t b(s)^\top \phi(s;t)b(s)ds \right)\\
  &\times \exp\biggl(\frac{1}{2}\beta(t;y)^\top \Sigma^\frac{1}{2}(I_{d_1}-\Sigma^\frac{1}{2}\phi(0;t)\Sigma^\frac{1}{2})^{-1}\Sigma^\frac{1}{2}\beta(t;y) \biggr)\\
  &\times \exp\biggl(\mu^\top \alpha(0,t;t)^\top y(t)\\
  &- \int_0^ty(s)^\top \left\{ a(s)+b(s)b(s)^\top \phi(s;t) \right\}\alpha(0,s;t)ds \mu+\frac{1}{2}\mu^\top \phi(0;t) \mu\biggr)\\
  &\times \exp\biggl( \frac{1}{2}y(t)^\top \gamma(t)y(t)
  +\frac{1}{2}\int_0^t y(s)^\top b(s)b(s)^\top y(s)ds\\
  &-\int_0^t y(s) b(s)b(s)^\top y(t) ds  \biggr),
  \end{align*}
  where
  \begin{align*}
    &\mu=E[X_0],~~\Sigma=\mathrm{Cov}(X_0,X_0),\\
    &\alpha(u,s;t)=\exp\left( \int_u^s \{a(r)+b(r)b(r)^\top \phi(r;t)\}dr \right),\\
    &\beta(t;y)= \alpha(0,t;t)^\top y(t)- \int_0^t \alpha(0,s;t)^\top \left\{ a(s)+b(s)b(s)^\top \phi(s;t) \right\}^\top y(s) ds\\
    &\qquad \qquad+ \phi(0;t)\mu,\\
    &\gamma(t)=\int_0^t \alpha(s,t;t)b(s)b(s)^\top \alpha(s,t;t)^\top ds.
  \end{align*}
\end{lemma}
\begin{proof}
  Due to Lemma \ref{lemma-measure-change}, it follows that
  \begin{align*}
    z_t(y)=&\exp\left( X_t^\top y(t)-\int_0^t y(s)^\top dX_s\right.\\
    &\left.+\frac{1}{2}\int_0^t b(s)^\top \phi(s;t)b(s)ds+\frac{1}{2}X_0^\top \phi(0;t) X_0 \right)\frac{d\mu_{\xi}}{d\mu_X}(X.)
  \end{align*}
  and
  \begin{align}
    \label{eq-z-y-1}\begin{split}
      E_Q[z_t(y)]=E\biggl[ \exp\biggl(&{\xi}_{t;t}^\top y(t)- \int_0^ty(s)^\top d{\xi}_{s;t}\\
     &+\frac{1}{2}\left. \left.\int_0^t b(s)^\top \phi(s;t)b(s)ds+\frac{1}{2}X_0^\top \phi(0;t) X_0\right) \right].
    \end{split}    
  \end{align}
  Furthermore, if we define $\{\overline{\xi}_{s}\}_{0\leq s\leq t}$ by
  \begin{align*}
    d{\overline{\xi}}_{s;t}=\left\{ a(s)+b(s)b(s)^\top \phi(s;t) \right\}{\overline{\xi}}_{s;t}ds+b(s)d\tilde{V}_s,~~{\overline{\xi}}_{0;t}=0,
  \end{align*}
  then we can write
  \begin{align*}
    \xi_{s,t}=\alpha(0,s;t)X_0+\overline{\xi}_{s,t}.
  \end{align*}
  Substituting this expression into (\ref{eq-z-y-1}), we obtain
  \begin{align}
    &E_Q[z_t(y)]\nonumber \\
     \label{eq3-7}\begin{split}
      =&\exp\left( \frac{1}{2}\int_0^t b(s)^\top \phi(s;t)b(s)ds \right)\\
     &\times E\biggl[ \exp\biggl(X_0^\top \alpha(0,t;t)^\top y(t)\\
     &\left. \left.- \int_0^ty(s)^\top \left\{ a(s)+b(s)b(s)^\top \phi(s;t) \right\}\alpha(0,s;t)ds X_0+\frac{1}{2}X_0^\top \phi(0;t) X_0\right) \right]\\
     &\times E\left[ \exp\left(\overline{\xi}_{t;t}^\top y(t)- \int_0^ty(s)^\top d\overline{\xi}_{s;t}\right) \right].
     \end{split}     
  \end{align}

  For the first expectation on the right-hand side, a straightforward calculation yields
  \begin{align}
    &E\biggl[ \exp\biggl(X_0^\top \alpha(0,t;t)^\top y(t)\nonumber\\
     &\left. \left.- \int_0^ty(s)^\top \left\{ a(s)+b(s)b(s)^\top \phi(s;t) \right\}\alpha(0,s;t)ds X_0+\frac{1}{2}X_0^\top \phi(0;t) X_0\right) \right]\nonumber\\
     \label{eq3-8}\begin{split}
      =&\sqrt{\det|I_{d_1}-\Sigma^\frac{1}{2}\phi(0;t)\Sigma^\frac{1}{2}|}\exp\biggl(\frac{1}{2}\beta(t;y)^\top \Sigma^\frac{1}{2}(I_{d_1}-\Sigma^\frac{1}{2}\phi(0;t)\Sigma^\frac{1}{2})^{-1}\Sigma^\frac{1}{2}\beta(t;y) \biggr)\\
     &\times \exp\biggl(\mu^\top \alpha(0,t;t)^\top y(t)\\
     &- \int_0^ty(s)^\top \left\{ a(s)+b(s)b(s)^\top \phi(s;t) \right\}\alpha(0,s;t)ds \mu+\frac{1}{2}\mu^\top \phi(0;t) \mu\biggr).
     \end{split}     
  \end{align}  
  Here, we used the negative semi-definiteness of $\phi(0;t)$.

  For the second expectation, since $\displaystyle \overline{\xi}_{t;t}^\top y(t)- \int_0^ty(s)^\top d\overline{\xi}_{s;t}$ is a Gaussian random variable with mean 0, it follows that 
  \begin{align}
    &E\left[ \exp\left(\overline{\xi}_{t;t}^\top y(t)- \int_0^ty(s)^\top d\overline{\xi}_{s;t}\right) \right]\nonumber\\
    =&\exp\left( \frac{1}{2}E\left[ \left( \overline{\xi}_{t;t}^\top y(t)- \int_0^ty(s)^\top d\overline{\xi}_{s;t} \right)^2 \right] \right)\nonumber\\
    \label{eq3-9}\begin{split}      
    =&\exp\biggl( \frac{1}{2}y(t)^\top E[\overline{\xi}_{t;t}\overline{\xi}_{t;t}^\top]y(t)
    +\frac{1}{2}\int_0^t y(s)^\top b(s)b(s)^\top y(s)ds\\
    &-\int_0^t y(s)^\top b(s)b(s)^\top y(t) ds  \biggr).
    \end{split}
  \end{align}
  Putting (\ref{eq3-7}), (\ref{eq3-8}) and (\ref{eq3-9}) together, we obtain the desired result.
\end{proof}

\section*{Acknowledgement}
The author is deeply grateful to N. Yoshida for his valuable advice and insightful discussions.

\section*{Funding}
This work was supported by Japan Science and Technology Agency CREST JP-MJCR2115, JSPS KAKENHI Grant Number JP24KJ0667, and RIKEN Special Postdoctoral Researcher Program.

\bibliographystyle{abbrvnat}

%\bibliography{paper-1-arxiv}
\bibliography{arxiv2}

\end{document}